\RequirePackage{fix-cm}

\documentclass[smallextended,envcountsect,envcountsame,numbook]{svjour3} 
\smartqed 

\makeatletter
\def\cl@chapter{}
\makeatother
\usepackage[utf8]{inputenc}
\usepackage{lmodern}
\usepackage[pdfpagelabels,colorlinks]{hyperref}
\usepackage{mathtools,amsfonts,amssymb,amscd}
\usepackage[misc]{ifsym}
\usepackage[capitalize,nameinlink]{cleveref}
\usepackage{autonum}
\usepackage[shortlabels]{enumitem}
\usepackage{lmodern}
\usepackage{amssymb}
\usepackage{subcaption}
\usepackage{cite}
\usepackage[pdftex,dvipsnames]{xcolor}
\usepackage[colorinlistoftodos,prependcaption,textsize=tiny]{todonotes}
\reversemarginpar

\journalname{Computational Optimization and Applications}

\def\re{\mathbb R}
\def\na{\mathbb N}

\def\lV{\left\lVert }
\def\rV{\right\rVert }
\def\lv{\left\lvert }
\def\rv{\right\rvert}

\DeclareMathOperator{\Id}{Id}
\DeclareMathOperator{\aff}{aff}
\DeclareMathOperator{\dist}{dist}
\DeclareMathOperator{\dom}{dom}
\DeclareMathOperator{\inte}{int}
\DeclareMathOperator{\bound}{bd}
\DeclareMathOperator{\epi}{epi}
\newcommand{\norm}[1]{\left\lVert#1\right\rVert}
\newcommand{\scal}[2]{\left\langle{#1},{#2}  \right\rangle}
\DeclareMathOperator{\circum}{circ}

\usepackage[left,mathlines]{lineno}
\usepackage{etoolbox} 
\newcommand*\linenomathpatchAMS[1]{%
  \expandafter\pretocmd\csname #1\endcsname {\linenomathAMS}{}{}%
  \expandafter\pretocmd\csname #1*\endcsname{\linenomathAMS}{}{}%
  \expandafter\apptocmd\csname end#1\endcsname {\endlinenomath}{}{}%
  \expandafter\apptocmd\csname end#1*\endcsname{\endlinenomath}{}{}%
}

\expandafter\ifx\linenomath\linenomathWithnumbers
  \let\linenomathAMS\linenomathWithnumbers
  \patchcmd\linenomathAMS{\advance\postdisplaypenalty\linenopenalty}{}{}{}
\else
  \let\linenomathAMS\linenomathNonumbers
\fi

\linenomathpatchAMS{gather}
\linenomathpatchAMS{multline}
\linenomathpatchAMS{align}
\linenomathpatchAMS{alignat}
\linenomathpatchAMS{flalign}

\begin{document}

\title{The circumcentered-reflection method achieves better rates than alternating projections~\thanks{RB was partially supported by the \emph{Brazilian Agency Conselho Nacional de Desenvolvimento Cient\'ifico e Tecnol\'ogico} (CNPq), Grants 304392/2018-9 and 429915/2018-7; \\ YBC was partially supported by the \emph{National Science Foundation} (NSF), Grant DMS -- 1816449.}}

\author{Reza Arefidamghani
\and Roger Behling  \and Yunier Bello-Cruz  \and 
Alfredo N. Iusem\and Luiz-Rafael Santos 
}

\institute{
  Reza Arefidamghani \and Alfredo N. Iusem  \at Instituto de Matem\'atica Pura e Aplicada\\
 Rio de Janeiro, RJ -- 22460-320,  Brazil \email{\{reza.arefidamghani, iusp\}@impa.br} 
\and 
  Roger Behling    \at School of Applied Mathematics, Funda\c{c}\~ao Get\'ulio Vargas \\ 
Rio de Janeiro, RJ -- 22250-900, Brazil. \email{rogerbehling@gmail.com}
    \and 
Yunier Bello-Cruz \at Department of Mathematical Sciences, Northern Illinois University. \\ 
 DeKalb, IL -- 60115-2828, USA. \email{yunierbello@niu.edu}
\and 
 Luiz-Rafael Santos \Letter \at Department of Mathematics, Federal University of Santa Catarina. \\ 
Blumenau, SC -- 88040-900, Brazil. \email{l.r.santos@ufsc.br}
}
\date{\today}

\titlerunning{CRM achieves better rates than MAP}
\authorrunning{Arefidamghani, Behling, Bello-Cruz, Iusem and Santos}

\maketitle

\begin{abstract}

We study the convergence rate  of the Circumcentered-Reflection Method (CRM) for solving
the convex feasibility problem and compare it with the Method of Alternating Projections (MAP). Under
an error bound assumption, we prove that both methods converge linearly, with asymptotic constants
depending on a parameter of the error bound, and that the one derived for CRM is strictly better than the 
one for MAP. Next, we analyze two classes of fairly generic examples. 
In the first one, the angle between the convex sets approaches zero near the intersection, so that
the MAP sequence converges sublinearly, but CRM still enjoys linear convergence. In the second class
of examples, the angle between the sets does not vanish and MAP exhibits its standard behavior, \emph{i.e.}, it 
converges linearly, yet, perhaps surprisingly, CRM attains superlinear convergence.


\end{abstract}

\section{Introduction}\label{s0}

We deal in this paper with the convex feasibility problem (CFP), 
defined as follows:
given closed and convex sets $K_1, \dots , K_m\subset\re^n$, find a point in
$\cap_{i=1}^mK_i$.

We study the \emph{Circumcentered-Reflection method} (CRM) for solving CFP under an error bound regularity condition. CRM and generalized circumcenters were introduced in \cite{Behling:2018,Behling:2018a}, and subsequently studied in \cite{Behling:2019,Behling:2020,Bauschke:2018a,Bauschke:2021,Bauschke:2020,Bauschke:2020c} with the geometrical appeal of accelerating classical projection/reflection based methods. We will see that a nonaffine setting embedded with an error bound condition provides, at least, linear convergence of CRM. Moreover, we present a quite general instance for which CRM converges superlinearly, opening a path for future research lines on circumcenters-type schemes. In addition to these contributions, we show that CRM  is faster than the famous \emph{Method of Alternating Projections }(MAP), even in the lack of an error bound. MAP has a vast literature  (see, for instance, \cite{Bauschke:1993,Kayalar:1988,Neumann:1950,Bauschke:2016a}) and concerns CPF for two sets, in principle. However, the discussion below allows us to apply both CRM and MAP to the multi-set CFP.

Two very well-known methods for CFP 
related to MAP are the \emph{Sequential Projection
Method} (SePM) and the \emph{Simultaneous Projection Method} (SiPM), which can be traced
back to \cite{Kaczmarz:1937} and \cite{Cimmino:1938} respectively, and are defined as follows. 
Consider the operators $\widehat P,
\overline P:\re^n\to\re^n$ given by $\widehat P\coloneqq P_{K_m}\circ\dots\circ P_{K_1}$,
$\overline P\coloneqq\frac{1}{m}\sum_{i=1}^mP_{K_i}$, where each $P_{K_i}:\re^n\to K_i$ is the orthogonal projection  
onto $K_i$.  Starting from an arbitrary $z \in\re^n$,
SePM and SiPM generate  sequences $(\hat{x}^k)_{k\in\na}$ and $(\bar{x}^k)_{k\in\na}$ given by $\hat{x}^{k+1}=\widehat P(\hat{x}^k)$,
$\bar{x}^{k+1}=\overline P(\bar{x}^k)$, respectively, where $\bar{x}^{0} = \hat{x}^{0} = z$. When  $\cap_{i=1}^mK_i\ne\emptyset$, 
the sequences generated by both
methods are known to be globally convergent to points belonging to $\cap_{i=1}^mK_i$,
\emph{i.e.}, to solve CFP. Under suitable assumptions, both methods have 
interesting convergence properties also in the infeasible case, \emph{i.e.},  when  
$\cap_{i=1}^mK_i =\emptyset$, but we will not deal with 
this case. See \cite{Bauschke:1996} for an in-depth study of these and other 
projections methods for CFP. 

An interesting relation between SePM and SiPM was
found by Pierra in \cite{Pierra:1984}. Consider the sets $\mathrm{K}\coloneqq K_1 \times\dots\times K_m\subset
\re^{nm}, \mathrm{U}\coloneqq\{(x,\dots,x)\mid x\in\re^m\}\subset\re^{nm}$. 
Apply SePM to the sets $\mathrm{K},\mathrm{U}$ in the product space $\re^{n m}$,
\emph{i.e.}, take $\mathrm{x}^{k+1}=P_{\mathrm{U}}\left(P_{\mathrm{K}}(\mathrm{x}^k)\right)$
starting from $\mathrm{x}^0\in \mathrm{U}$. Clearly, $\mathrm{x}^k$ belongs to $\mathrm{U}$
for all $k\in\na$, so that we may write $\mathrm{x}^k=(x^k,\ldots,x^k)$ with $x^k\in\re^n$. It was proved in 
\cite{Pierra:1984} that $x^{k+1}=\overline P(x^k)$, \emph{i.e.}, a step of SePM applied to two convex sets in
the product space $\re^{n\times m}$ is equivalent to a step of SiPM in the original space $\re^n$.
Thus, SePM with just two sets plays a sort of special role and, therefore, carries a name of its own, 
namely  MAP. 
Observe that in the equivalence above one of the two sets in the product space, 
namely $\mathrm{U}$, is a linear subspace.
This fact is essential for the convergence of CRM 
applied for solving CFP; see \cite{AragonArtacho:2019,Behling:2020}. 

Let us start to focus on the alleged acceleration effect of CRM with respect to MAP. There is abundant numerical evidence
of this effect (see \cite{Behling:2018a,Behling:2020,Dizon:2019,Behling:2020b}); in this paper, we will present some  
analytical evidence, which strengthens the results from~\cite{Behling:2020}. In view of Pierra’s reformulation~\cite{Pierra:1984}, the general CFP can be seen as a specific convex-affine intersection problem and since both CRM and MAP converge for the general convex-affine intersection problem, from now on, we seek a point common to a given closed convex set $K\subset\re^n$  and an affine manifold $U\subset\re^n$.

For finding a point in $K\cap U\neq \emptyset$,  MAP and CRM iterate by means of the operators  $T=P_U\circ P_K$   and  $C(\cdot)=\circum(\cdot, R_K(\cdot), R_U(R_K(\cdot)))$, respectively, where   $R_K=2P_K-\Id$ and $R_U=2P_U-\Id$ are the reflection operators over $K$ and $U$, respectively. 
For a point $x\in \re^n$, $C(x)$, when exists, is the point equidistant to $x,R_K(x)$  and $R_U(R_K(x))$  that lies in the affine manifold determined by the latter three points.

A first
result in the analytical study of the acceleration effect of CRM over MAP was derived in \cite{Behling:2020}, where it was proved
that, for all $x\in U$, $C(x)$ is well-defined and  $\dist(C(x),K\cap U)\le \dist(T(x),K\cap U)$, where $\dist$ stands for the Euclidean distance. The previous inequality means that the point obtained after a CRM step is closer to 
(or at least
no farther from) $K\cap U$ than the one obtained after a MAP step from the same point. 
This local (or myopic) acceleration does not imply immediately
that the CRM sequence converges faster than the MAP one. In order to show global acceleration, we will focus on special
situations where the convergence rate of the MAP can be precisely established. 

MAP is known to be linearly convergent
in several special situations, \emph{e.g.}, when both $K$ and $U$ are affine manifolds (see \cite{Kayalar:1988}) or when $K\cap U$ has
nonempty interior (see \cite{Bauschke:1993}). In \cref{s2} we will consider another such case, 
namely when a certain so-called {\it error bound} (EB from now on)
holds, meaning that there exists $\omega\in (0,1)$ such that $\dist(x, K)\ge\omega \dist(x,K\cap U)$ for all $x\in U$. 
This error bound resembles the regularity conditions presented in \cite{Bauschke:1993,Bauschke:1996,Behling:2017}. 
We will prove that in this case both the MAP and the CRM sequences converge linearly, with asymptotic constants bounded by
$\sqrt{1-\omega^2}$ for MAP, and by the strictly better bound $\sqrt{{1-\omega^2}}/\sqrt{{1+\omega^{2}}}$ for CRM, thus showing 
that under EB, CRM is faster than MAP. For the case of MAP, linear convergence under the error bound condition with this asymptotic constant is already known (see, for instance, \cite[Corollary 3.14]{Bauschke:1993}) even if $U$ is not affine;  
we present it for the sake of completeness.
Then, in \cref{s3} we will exhibit two rather generic families of
examples where CRM converges indeed faster than MAP. In the first one, $K\subset\re^{n+1}$ will be the epigraph of a
convex function $f:\re^n\to\re\cup\{+\infty\}$ and $U\subset\re^{n+1}$ a support hyperplane of $K$. We will show that
in this situation, under adequate assumptions on $f$, the MAP sequence converges sublinearly, while the CRM 
sequence converges linearly, and we will give as well an explicit bound for the asymptotic constant of the CRM sequence. Also, we will present
a somewhat pathological example for which both the MAP sequence and the CRM one converge sublinearly. 
In the second family, $K$ will still be the epigraph of a convex $f$, but $U$ will not be a supporting hyperplane of $K$; 
rather it will intersect the interior of $K$. In this case, under not too demanding assumptions on $f$, the MAP sequence converges 
linearly (we will give an explicit
bound of its asymptotic constant), while CRM converges superlinearly. These results firmly corroborate the already established numerical evidence in \cite{Behling:2020} of the superiority of CRM over MAP.

\section{Preliminaries}\label{s1}
 
We recall first the definition of Q-linear and R-linear convergence.

\begin{definition}\label[definition]{d1}
Let $(y^k)_{k\in\na}\subset\re^n$ be a convergent sequence to $y^*$. Assume that $y^k\ne y^*$ for all $k\in\na$.
Define \begin{equation}q\coloneqq \limsup_{k\to\infty}\frac{\lV y^{k+1}-y^*\rV}{\lV y^k-y^*\rV}, \quad   \mbox{and}\quad 
r\coloneqq \limsup_{k\to\infty}\lV y^k-y^*\rV^{1/k}.\end{equation} Then, the convergence of $(y^k)_{k\in\na}$ is
\begin{enumerate}[(i), format=\bf,leftmargin=*,align=right, widest=iii]
	\item \emph{Q-superlinearly} if $q=0$;
	\item \emph{Q-lineary} if $q\in(0,1)$;
	\item \emph{Q-sublinearly} if $q\ge 1$;
	\item \emph{R-superlinearly} if $r=0$;
	\item \emph{R-linearly} if $r\in(0,1)$;
\item \emph{R-sublinearly} if $r\ge 1$. 
\end{enumerate}

The values
$q,r$ are called \emph{asymptotic constants} of $(y^k)_{k\in\na}$.
\end{definition}

It is well known that Q-linear convergence implies R-linear convergence (with the same asymptotic constant), but the
converse statement does not hold true.

We remind now the notion of Fej\'er monotonicity in $\re^{n}$.

\begin{definition}
A sequence $(y^k)_{k\in\na}$ is\emph{ Fej\'er monotone} with respect to a set $M$ when 
$\lV y^{k+1}-y\rV\le\lV y^k-y\rV$, for all $y\in M$.
\end{definition}

\begin{proposition}\label[proposition]{p1}
If $(y^k)_{k\in\na}$ is Fej\'er monotone with respect to $M$ then 
\begin{enumerate}[(i), format=\bf,leftmargin=*,align=right, widest=ii]
\item $(y^k)_{k\in\na}$ is bounded;
\item  if a cluster point $\bar y$ of $(y^k)_{k\in\na}$ belongs to $M$, then $\displaystyle\lim_{k\to\infty}y^k=\bar y$.
\end{enumerate}
\end{proposition} 

\begin{proof} 
See Theorem 2.16 in \cite{Bauschke:1996}.

\qed \end{proof}

We end this section with the main convergence results for MAP and CRM.

\begin{proposition}\label[proposition]{p2} Take closed and convex sets $K_1,K_2\subset\re^n$ such that $K_1\cap K_2\ne\emptyset$.
Let $P_{K_1}, P_{K_2}$ be the orthogonal projections onto $K_1,K_2$ respectively. Consider the sequence $(z^k)_{k\in\na}$ generated
by MAP starting from any $z^0\in\re^n$, \emph{i.e.}, $z^{k+1}=P_{K_2}(P_{K_1}(z^k))$. Then $(z^k)_{k\in\na}$ is Fej\'er monotone with 
respect to $K_1\cap K_2$ and converges to a point $z^*\in K_1\cap K_2$.
\end{proposition}

\begin{proof} 
See \cite[Theorem 4]{Cheney:1959}.

\qed \end{proof} 

Let us present the formal definition of the circumcenter. 
\begin{definition}\label[definition]{def:circum}
Let $x,y,z\in\re^n$ be given. The circumcenter $\circum(x,y,z)\in\re^n$ is a point satisfying
\begin{enumerate}[(i), format=\bf,leftmargin=*,align=right, widest=ii]
	\item   $\norm{\circum(x,y,z) - x}=\norm{\circum(x,y,z) - y}=\norm{\circum(x,y,z) - z}$ and,
	\item  $\circum(x,y,z)\in \aff\{x,y,z\}\coloneqq \{w\in\re^n \mid w=x+\alpha (y-x)+\beta (z-x),\;\alpha,\beta\in\re\}$.
\end{enumerate}
\end{definition}

The point $\circum(x,y,z)$ is well and uniquely defined if the cardinality of the set $\{x,y,z\}$ is one or two. In 
the case in which the three points are all distinct,  $\circum(x,y,z)$ is well and uniquely defined  only if $x$, $y$ and $z$ 
are not collinear. For more general notions, definitions and results on circumcenters see \cite{Behling:2020b,Behling:2018}.

Consider now a closed convex set $K\subset\re^n$ and an affine manifold $U\subset\re^n$. 
Let $P_K, P_U$ be the orthogonal projections onto $K,U$ respectively. Define $R_K,R_U, T, C:\re^n\to\re^n$ as
\begin{equation}\label{e1}
R_K=2P_K-\Id, \quad R_U=2P_U-\Id,\quad T = P_U \circ P_K, \quad C(\cdot)=\circum(\cdot,R_K(\cdot), R_U(R_K(\cdot))).
\end{equation}

\begin{proposition}\label[proposition]{prop:nonexpansiveProj}
Let $K\subset \re^n$ be a nonempty closed convex set. Then, the orthogonal projection $P_K$ onto $K$ is firmly nonexpansive, that is, for all  $x,y\in \re^n$ we have
\[
\label{eq:firmlynonexpansive}
\lV P_K(x)-P_K(y)\rV^2\le\lV x-y\rV^2-\lV (x-P_K(x)) - (y - P_K(y))\rV^2.
\]
\end{proposition}
\begin{proof}
See \cite[Theorem 4.16]{Bauschke:2017}.
\end{proof}
\begin{proposition}\label[proposition]{p3}
Assume that $K\cap U\ne\emptyset$. Let $(x^k)_{k\in\na}$ be the sequence generated by CRM starting from any
$x^0\in U$, \emph{i.e.}, $x^{k+1}=C(x^k)$. Then,
\begin{enumerate}[(i), format=\bf,ref=3(\roman*),leftmargin=*,align=right, widest=iii]

	\item \label{p3ii}  for all $x\in U$,  we have that $C(x)$ is well defined  and belongs to $U$;

	\item \label{p3i} for all $x\in U$, it holds that  $\lV C(x) -y\rV\le\lV T(x)-y\rV$, for any $y\in K\cap U$;
	
	\item \label{p3iii} $(x^k)_{k\in\na}$ is Fej\'er monotone with respect to $K\cap U$;
	\item \label{p3iv}  $(x^k)_{k\in\na}$ converges to a point in $K\cap U$.
\end{enumerate}
\end{proposition}

\begin{proof} All these results can be found in \cite{Behling:2020}: (i) is proved in Lemma 3, 
(ii) in Theorem 2, and (iii) and (iv) in Theorem 1.
\qed 
\end{proof} 

\section{Linear convergence of MAP and CRM under an error bound assumption}\label{s2}

We start by introducing an assumption on a pair of convex sets $K,K'\subset\re^n$, 
denoted as EB (as in Error Bound), which will ensure linear convergence of MAP and CRM.

\begin{enumerate}[label=EB.,font=\bfseries,leftmargin=*,ref=EB]
\item \label[assumption]{EB} $K\cap K'\ne\emptyset$ and 
there exists $\omega\in (0,1)$ such that $\dist(x,K)\ge\omega \dist(x,K\cap K')$ for all $x\in K'$.
\end{enumerate}

Now we consider a closed and convex set $K\subset\re^n$ and an affine manifold $U\subset\re^n$.
Assuming that $K,U$ satisfy Assumption \ref{EB}, we will prove linear convergence of the sequences $(z^k)_{k\in\na}$ 
and $(x^k)_{k\in\na}$ generated 
by MAP and CRM, respectively. We start by proving that, for both methods, both distance sequences 
$(\dist(z^k, K\cap U))_{k\in\na}$  and $(\dist(x^k, K\cap U))_{k\in\na}$ converge linearly to $0$,
which will be a corollary of the next proposition.

\begin{proposition}\label[proposition]{p4}
Assume that $K,U$ satisfy \ref{EB}. Consider $T,C:\re^n\to\re^n$ as in \eqref{e1}. Then, for all $x\in U$,
\begin{equation}\label{e2}
(1-\omega^2)\lV x-P_{K\cap U}(x)\rV^2\ge\lV T(x)-P_{K\cap U}(T(x))\rV^2\ge\lV C(x)-P_{K\cap U}(C(x))\rV^2,
\end{equation}
with $\omega$ as in Assumption \ref{EB}.
\end{proposition}

\begin{proof}
It follows easily from \cref{prop:nonexpansiveProj} that
\begin{equation}\label{e4}
\lV P_K(x)-y\rV^2\le\lV x-y\rV^2-\lV x-P_K(x)\rV^2  
\end{equation}
for all $x\in\re^n$ and all $y\in K\cap U\subset K$. Invoking again \cref{prop:nonexpansiveProj}, we get from \eqref{e4}
\begin{align}
\lV T(x)-y\rV^2&=\lV P_U(P_K(x))- y\rV^2\le\lV P_K(x)-y\rV^2-\lV P_U(P_K(x))-P_K(x)\rV^2
\\ &\le
\lV x-y\rV^2-\lV x-P_K(x)\rV^2-\lV P_U(P_K(x))-P_K(x)\rV^2
\\
& \le\lV x-y\rV^2-\lV x-P_K(x)\rV^2=
\lV x-y\rV^2-\dist^2(x,K)\\
& \le\lV x-y\rV^2-\omega^2 \dist^2(x,K\cap U)\label{e5}
\end{align}
for all $x\in U, y\in K\cap U$, using Assumption \ref{EB} in the last inequality. Take now $y=P_{K\cap U}(x)$.
Then, in view of \eqref{e5},
\begin{align}
\lV C(x)-P_{K\cap U}(C(x))\rV^2& \le\lV C(x)-P_{K\cap U}(T(x))\rV^2\\ &
\le\lV T(x)-P_{K\cap U}(T(x))\rV^2
\le\lV T(x)-P_{K\cap U}(x)\rV^2\\
& \le \lV x-P_{K\cap U}(x)\rV^2-\omega^2 \dist^2(x,K\cap U) \\ 
&= (1-\omega^2)\lV x-P_{K\cap U}(x)\rV^2,\label{e6}
\end{align}
using the definition of $P_{K\cap U}$ in the first and the third inequality and  Proposition~\ref{p3}(ii)
in the second inequality. Note that \eqref{e2} follows immediately from \eqref{e6}.

\qed \end{proof}

\begin{corollary}\label[corollary]{c1}
Let $(z^k)_{k\in\na}$ and $(x^k)_{k\in\na}$ be the sequences generated by MAP and CRM starting at any $z^0\in \re^n$ and 
any $x^0\in U$, respectively. If $K,U$ satisfy Assumption
\ref{EB}, then the sequences  $(\dist(z^k,K\cap U))_{k\in\na}$ and $(\dist(x^k,K\cap U))_{k\in\na}$ converge Q-linearly to $0$, and the asymptotic constants
are bounded above by $\sqrt{1-\omega^2}$, with $\omega$ as in Assumption \ref{EB}.
\end{corollary} 

\begin{proof}
In view of the definition of $P_{K\cap U}$, \eqref{e2} can be rewritten as
\begin{equation}\label{e7}
(1-\omega^2)\dist^2(x,K\cap U)\ge \dist^2(T(x),K\cap U)\ge \dist^2(C(x),K\cap U),
\end{equation}
for all $x\in U$.
Since $z^{k+1}=T(z^k)$, we get from the first inequality in \eqref{e7}, 
\[(1-\omega^2)\dist^2(z^k,K\cap U)\ge \dist^2(z^{k+1},K\cap U),\] using the fact that $(z^k)_{k\in\na}\subset U$. Hence 
\begin{equation}\label{e8}
\frac{\dist(z^{k+1},K\cap U)}{\dist(z^k,K\cap U)}\le \sqrt{1-\omega^2}.
\end{equation}
By the same token, using
the second inequality in \eqref{e7} and \cref{p3}(ii), we get
\begin{equation}\label{e9}
\frac{\dist(x^{k+1},K\cap U)}{\dist(x^k,K\cap U)}\le \sqrt{1-\omega^2}.
\end{equation}
 The inequalities  in \eqref{e8} and \eqref{e9} imply the result.

\qed \end{proof}

We remark that the result for MAP holds when $U$ is any closed and convex set, 
not necessarily an affine manifold. We need $U$ to be an affine manifold in \cref{p3}
(otherwise, $(x^k)_{k\in\na}$ may even diverge), but this proposition is used in our proofs only
when the CRM sequence is involved.

Next, we  show that, under Assumption~\ref{EB}, CRM achieves a linear rate with an asymptotic constant better 
than the one given in \cref{c1}.

\begin{proposition}\label[proposition]{prop:CRMBetterRate}
  Let $(x^k)_{k\in\na}$ be the sequence generated by  CRM starting at any $x^0\in U$. 
	If $K,U$ satisfy Assumption \ref{EB}, then the sequence  $(\dist(x^k,K\cap U))_{k\in\na}$ Q-converges to $0$ 
	with the asymptotic constant bounded above by 
  \(\sqrt{\dfrac{1-\omega^2}{1+\omega^{2}}},
  \) where $\omega$ is as in Assumption \ref{EB}.
\end{proposition}

\begin{proof}
Take $y^*\in K\cap U$ and $x\in U$. Note that 
\begin{align}
\dist^2(x,K)&=\lV x-P_K(x)\rV^2 
\\
&\le\lV x-y^*\rV^2-\lV P_K(x)-y^*\rV^2\\
&=\lV x-y^*\rV^2-\lV P_K(x)-P_K(y^*)\rV^2\\
&\le\lV x-y^*\rV^2-\lV P_U(P_K(x))-P_U(P_K(y^*))\rV^2-\lV P_U(P_K(x))-P_K(x)\rV^2\\
&=\lV x-y^*\rV^2-\lV P_U(P_K(x))-y^*\rV^2-\lV P_U(P_K(x))-P_K(x)\rV^2,\label{eq*}
\end{align}
using the definition of orthogonal projection onto $K$, the fact that $y^*\in K$ and \cref{prop:nonexpansiveProj} in the first inequality, and again 
\cref{prop:nonexpansiveProj} regarding $U$ in the second inequality.

Now, we will invoke some results from \cite{Behling:2020} to prove that $C(x)$ is indeed the orthogonal
projection of $x$ onto the intersection of $U$ with the halfspace $H_x^+\coloneqq \{y\in \re^n \mid \scal{y-P_K(x)}{ x-P_K(x)}\le 0\}$ containing $K$. In \cite[Lemma 3]{Behling:2020}  it is proved that  
$C(x)=P_{H_x\cap U}(x)$ with $H_x\coloneqq \{y\in \re^n \mid \scal{y-P_K(x)}{ x-P_K(x)}= 0\}$. Using the arguments employed at the beginning  of the  proof of \cite[Lemma 5]{Behling:2020}, we get \[
C(x) = P_{H_x\cap U}(x) = P_{H_x^+\cap U}(x).
\]
Hence, the above equality  and the fact that $y^*\in K\cap U \subset H_x^+\cap U$ imply $\scal{y^*-C(x)}{x-C(x)}\le 0$ and since $x$, $P_U(P_K(x))$ and $C(x)$ are collinear (see  \cite[Eq. (7)]{Behling:2020}), 
we get $\scal{y^*-C(x)}{P_U(P_K(x))-C(x)}\le 0$. Thus,
\[\label{fifi} 
\|P_U(P_K(x))-y^*\|^2\ge \|C(x)-y^*\|^2+\|C(x)-P_U(P_K(x))\|^2.
\]
Now, \eqref{fifi} and \eqref{eq*} imply
\begin{align}
\dist^2(x,K)&\le  \lV x-y^*\rV^2-\lV C(x)-y^*\rV^2-\lV C(x)-P_U(P_K(x))\rV^2-\lV P_U(P_K(x))-P_K(x)\rV^2\\
&=\lV x-y^*\rV^2-\lV C(x)-y^*\rV^2-\lV C(x)-P_K(x)\rV^2\\
&\le  \lV x-y^*\rV^2-\lV C(x)-y^*\rV^2-\dist^2(C(x), K)\\
&\le  \lV x-y^*\rV^2-\dist^2(C(x),K\cap U) -\dist^2(C(x), K),
\end{align}
using the definition of the distance in the last two inequalities. Now, taking $y^*=P_{K\cap U}(x)$ and using the error 
bound condition for $x$ and $C(x)$, we obtain
\begin{align}
\omega^2\dist^2(x,K\cap U)\le \dist^2(x,K)&\le  \dist^2(x,K\cap U)-\dist^2(C(x),K\cap U) -\omega^2\dist^2(C(x), K\cap U)\\
&=  \dist^2(x,K\cap U)- (1+\omega^2)\dist^2(C(x), K\cap U).\label{fufu}
\end{align}
Rearranging \eqref{fufu}, we get $(1+\omega^2)\dist^2(C(x), K\cap U)\le (1-\omega^2)\dist^2(x,K\cap U) $ and,  
since $x^{k+1}=C(x^k)$, we have
\begin{align}
\frac{\dist(x^{k+1}, K\cap U)}{\dist(x^{k},K\cap U)}\le \sqrt{ \frac{1-\omega^2}{1+\omega^2}}, 
\end{align}
which implies the result.
\qed \end{proof}

\cref{p4,prop:CRMBetterRate} do not entail immediately that the sequences $(x^k)_{k\in\na}, (z^k)_{k\in\na}$ themselves
converge linearly; a sequence $(y^k)_{k\in\na}\subset\re^n$ may converge to a point $y\in M\subset\re^n$,
in such a way that $(\dist(y^k, M))_{k\in\na}$ converges linearly to $0$ but $(y^k)_{k\in\na}$ itself converges sublinearly.
Take for instance $M=\{(s,0)\in\re^2\}$, $y^k=\left(1/k,2^{-k}\right)$. This sequence converges
to $0\in M$, $\dist(y^k,M)=2^{-k}$ converges linearly to $0$ with asymptotic constant equal to $1/2$, but
the first component of $y^k$ converges to $0$ sublinearly, and hence the same holds for the sequence
$(y^k)_{k\in\na}$. The next lemma, possibly of some interest on its own, establishes that this situation
cannot occur when $(y^k)_{k\in\na}$ is Fej\'er monotone with respect to $M$. The result below is similar to \cite[Theorem 5.12]{Bauschke:2017}, however we include its proof for the sake of completeness.

\begin{lemma}\label[lemma]{l1} 
Consider a nonempty closed convex set $M\subset\re^n$ and  $(y^k)_{k\in\na}\subset\re^n$. Assume that $(y^k)_{k\in\na}$ is Fej\'er
monotone with respect to $M$, and that $(\dist(y^k,M))_{k\in\na}$ converges R-linearly to $0$. Then $(y^k)_{k\in\na}$ converges 
R-linearly to some point $y^*\in M$, with asymptotic constant bounded above by the 
asymptotic constant of $(\dist(y^k,M))_{k\in\na}$.
\end{lemma}

\begin{proof}
Fix $k\in\na$ and note that the Fej\'er monotonicity hypothesis implies that, for all $j\ge k$, 
\begin{equation}\label{e10}
\lV y^j-P_M(y^k)\rV\le\lV y^k-P_M(y^k)\rV=\dist(y^k, M).
\end{equation}
By \cref{p1}(i), $(y^k)_{k\in\na}$ is bounded. Take any cluster point $\bar y$ of $(y^k)_{k\in\na}$.
Taking limits with $j\to\infty$ in \eqref{e10} along a subsequence $(y^{k_j})_{j\in\na}$ of $(y^k)_{k\in\na}$ converging
to $\bar y$, we get that $\lV\bar y-P_M(y^k)\rV\le \dist(y^k,M)$. Since $\lim_{k\to\infty}\dist(y^k,M)=0$,
we conclude that $(P_M(y^k))_{k\in\na}$ converges to $\bar y$, so that there exists a unique cluster point, say $y^*$.
Therefore, $\lim _{k\to\infty}y^k=y^*$, and hence $\lV y^*-P_M(y^k)\rV\le \dist(y^k,M)$. Since 
$y^*=\lim_{k\to\infty}P_M(y^k)$, we conclude that $y^*\in M$.
Observe further that
\begin{equation}\label{e11}
\lV y^k-y^*\rV\le\lV y^k-P_M(y^k)\rV+\lV P_M(y^k)-y^*\rV=\dist(y^k,M)+\lV y^*-P_M(y^k)\rV\le 2\dist(y^k,M).
\end{equation}
Taking $k$th-root and then $\limsup$ with $k\to\infty$ in \eqref{e11}, and using the R-linearity hypothesis,
\begin{align}
\limsup_{k\to\infty}\lV y^k-y^*\rV^{1/k}& \le\limsup_{k\to\infty}2^{1/k}\dist(y^k, M)^{1/k}\\
&= 
\limsup_{k\to\infty}\dist(y^k,M)^{1/k}<1, 
\end{align}
establishing both that $(y^k)_{k\in\na}$ converges R-linearly to $y^*\in M$ and the statement on the asymptotic constant.

\qed \end{proof}

With the help of \cref{l1}, we prove next  the R-linear convergence of the MAP and CRM sequences
under Assumption \ref{EB}, and give bounds for their asymptotic constants.

\begin{theorem} \label{t1}
Consider a closed and convex set $K\subset\re^n$ and an affine manifold $U\subset \re^n$. Assume that
$K,U$ satisfy Assumption \ref{EB}. Let $(z^k)_{k\in\na}, (x^k)_{k\in\na}$ be the sequences generated by MAP and CRM, respectively, 
starting from arbitrary points $z^0\in\re^n, x^0\in U$. Then both sequences $(z^k)_{k\in\na}$ and $(x^k)_{k\in\na}$ converge R-linearly to
points in $K\cap U$,
and the asymptotic constants are bounded above by $\sqrt{1-\omega^2}$ for MAP, and by $\sqrt{\dfrac{1-\omega^2}{1+\omega^{2}}}$ 
for CRM, with $\omega$ as in Assumption \ref{EB}.
\end{theorem}  
   
\begin{proof} In view of  \cref{p2,p3}(iii) and (iv), both 
sequences are Fej\'er monotone with respect to ${K\cap U}$ and converge to points in ${K\cap U}$. By \cref{c1},
both sequences  $(\dist(z^k,{K\cap U}))_{k\in\na}$ and $(\dist(x^k,{K\cap U}))_{k\in\na}$ are Q-linearly
convergent to $0$, and henceforth R-linearly convergent to $0$.  \cref{c1} shows that the asymptotic constant of 
the sequence $(\dist(z^k,{K\cap U}))_{k\in\na}$ is 
bounded above by $\sqrt{1-\omega^2}$, and \cref{prop:CRMBetterRate} establishes that the asymptotic constant of the sequence
$(\dist(x^k,{K\cap U}))_{k\in\na}$ 
is bounded above by $\sqrt{{1-\omega^2}}/\sqrt{{1+\omega^{2}}}$. Finally, by \cref{l1}, both 
$(z^k)_{k\in\na}$  and $(x^k)_{k\in\na}$ are R-linearly
convergent, with the announced bounds for their asymptotic constants.
\qed \end{proof}

We remark that, in view of \cref{t1}, the upper bound for the asymptotic constant of the CRM sequence is substantially 
better than the one for the MAP sequence.  Note that the CRM bound reduces the MAP one by a factor of $\sqrt{1+\omega^2}$, 
which increases up to $\sqrt{2}$ when $\omega$ approaches $1$.

\section{Two families of examples for which CRM is much faster than MAP}\label{s3}

We will present now two rather generic families of examples for which CRM is faster than MAP.
In the first one, MAP converges sublinearly while CRM converges linearly; in the
second one, MAP converges linearly and CRM converges superlinearly. 

In both families, we work in $\re^{n+1}$. $K$ will be the epigraph of a proper convex function 
$f:\re^n\to\re\cup\{+\infty\}$, and $U$ the hyperplane $\{(x,0)\mid  x\in\re^n\}\subset\re^{n+1}$. From now on, we 
consider $f$  to be continuously differentiable in $\inte(\dom(f))$, where  
$\dom(f)\coloneqq \{x\in\re^n\mid  f(x)<+\infty\}$ and  $\inte(\dom(f))$ is its topological interior. 
Next, we make the following assumptions on
$f$:
\begin{enumerate}[label=A\arabic*.,font=\bf,leftmargin=*,ref=A\arabic*]
	\item \label{A2} $0\in \inte(\dom(f))$ is the unique minimizer of $f$. 
	\item \label{A3} $\nabla f(0)=0$.
	\item \label{A4} $f(0)=0$.
\end{enumerate}

We will show that under these assumptions, MAP always converges sublinearly, 
while, adding an additional hypothesis, CRM converges linearly.
  
Note that under hypotheses \labelcref{A2,A3,A4}, $0\in\re^{n}$ is the unique 
zero of $f$ and hence $K\cap U=\{(0,0)\}\subset \re^{n+1}$. In view of   \cref{p2,p3}, the sequences
generated by MAP and CRM, with arbitrary initial points in $\re^{n+1}$ and $U$, respectively, 
both converge to $(0,0)$, and are Fej\'er monotone with respect to $\{(0,0)\}$, so that, in view of \ref{A2}, for large enough $k$
the iterates of both sequences belong to $\inte(\dom(f))\times \re$. 
 We take now any point $(x,0)\in U$, with $x\ne 0$ and proceed to
compute $P_K(x,0)$. Since $(x,0)\notin K$ (because $x\ne 0$ and $K\cap U=\{(0,0)\}$), $P_K(x,0)$ must belong
to the boundary of $K$, \emph{i.e.}, it must be of the form $(u,f(u))$, and $u$ is determined by minimizing $\lV (x,0)-(u,f(u))\rV^2$,
so that $u-x +f(u)\nabla f(u) =0$, or equivalently 
\begin{equation}\label{e12}
x=u+f(u)\nabla f(u).
\end{equation}
Note that since $x\ne 0$, $u\ne 0$ by \ref{A4}. With the notation of \cref{s2} and bearing in mind that $T$ and $C$ are the MAP and CRM operators defined in \eqref{e1},  it is easy to check that
\[\label{e13}
\begin{aligned}
P_K(x,0) &=(u,f(u)),\text{ and } \\ 
T(x,0)&=P_U(P_K(x,0))= (u,0),
\end{aligned}
\]
with $u$ as in \eqref{e12}. Moreover, 
\[ \begin{aligned}
R_K(x,0)&=(2u-x,2f(u)), \\
P_U(R_K(x,0))&= (2u-x,0), \text{ and } \\
  R_U(R_K(x,0))&=(2u-x,-2f(u)).
\end{aligned}
\]
Next we compute $C(x,0)=\circum((x,0),R_K(x,0), R_U(R_K(x,0)))$. Suppose that $C(x,0)=(v,s)$.
The conditions $\lV (v,s)-(x,0)\rV=\lV (v,s)-R_K(x,0)\rV=\lV (v,s)-R_U(R_K(x,0))\rV$ give rise to two
quadratic equations whose solution is
\begin{equation}\label{e14}
s=0,\qquad v=u-\left[\frac{f(u)}{\lV x-u\rV}\right]^2(x-u)=u -\frac{f(u)}{\lV\nabla f(u)\rV^2}\nabla f(u),
\end{equation}
using \eqref{e12} in the last equality. 

We proceed to compute the quotients $\lV T(x,0)-0\rV/\lV (x,0)-0\rV$,
$\lV C(x,0)-0\rV/\lV (x,0)-0\rV$. Since both the MAP and the CRM sequences converge to $0$, 
these quotients are needed for determining 
their convergence rates. In view of \eqref{e13} and \eqref{e14}, these quotients reduce to $\lV u\rV/\lV x\rV,
\lV v\rV/\lV x\rV$. We state the result of the computation of these quotients in the next proposition.

\begin{proposition}\label[proposition]{p5} 
Take $(x,0)\in U$ with $x\ne 0$. Let $T(x,0)=(u,0)$ and $C(x,0)=(v,0)$. Then,
\begin{equation}\label{e15}
\frac{\lV T(x,0)\rV}{\lV (x,0)\rV}=\frac{\lV u\rV}{\lV x\rV}=\frac{1}{\lV \bar u+\dfrac{f(u)}{\lV u\rV}\nabla f(u)\rV}
\end{equation}
with $\bar u=u/\lV u\rV$, 
\begin{equation}\label{e16}
\left[\frac{\lV C(x,0)\rV}{\lV T(x,0)\rV}\right]^2=\left[\frac{\lV v\rV}{\lV u\rV}\right]^2\le 
1-\left[\frac{f(u)}{\lV u\rV\,\lV \nabla f(u)\rV}\right]^2,
\end{equation}
and
\begin{equation}\label{ee16}
\left[\frac{\lV C(x,0)\rV}{\lV (x,0)\rV}\right]^2\le
\left[1-\left(\frac{f(u)}{\lV u\rV\,\lV \nabla f(u)\rV}\right)^2\right]\left[\frac{\lV u\rV}{\lV x\rV}\right]^2.
\end{equation}
\end{proposition} 

\begin{proof}
In view of \eqref{e12},
\[
\frac{\lV u\rV}{\lV x\rV}=\frac{\lV u\rV}{\lV u+f(u)\nabla f(u)\rV}
\] 
and \eqref{e15} follows by dividing  the numerator and the denominator by $\lV u\rV$.

We proceed to establish \eqref{e16}. In view of \eqref{e14}, we have
\begin{align}
\lV v\rV^2&=\lV u -\frac{f(u)}{\lV\nabla f(u)\rV^2}\nabla f(u)\rV^2\\&=\lV u\rV^2+\left[\frac{f(u)}{\lV\nabla f(u)\rV}\right]^2-
2\frac{f(u)}{\lV\nabla f(u)\rV^2}\langle\nabla f(u),u\rangle\\&\le
\lV u\rV^2+\left[\frac{f(u)}{\lV\nabla f(u)\rV}\right]^2-2\left[\frac{f(u)}{\lV\nabla f(u)\rV}\right]^2\\&=
\lV u\rV^2-\left[\frac{f(u)}{\lV\nabla f(u)\rV}\right]^2,\label{e17}
\end{align}
using the gradient inequality $\langle\nabla f(u),u\rangle\ge f(u)$, which holds because $f$ is convex and $f(0)=0$.  
Now, \eqref{e16} follows by dividing \eqref{e17} by $\lV u\rV^2$.
Finally, \eqref{ee16} follows by multiplying \eqref{e16} by $\dfrac{\lV u\rV^2}{\|x\|^2}=\dfrac{\lV T(x,0)\rV^2}{\|(x,0)\|^2}$.

\qed \end{proof}

Next we compute the limits with $x\to  0$ of the quotients in \cref{p5}.

\begin{proposition}\label[proposition]{p6} 
Take $(x,0)\in U$ with $x\ne 0$. Let $T(x,0)=(u,0)$ and $C(x,0)=(v,0)$. Then,
\begin{equation}\label{e18}
\limsup_{x\to 0}\frac{\lV T(x,0)\rV}{\lV (x,0)\rV}=\lim_{x\to 0}\frac{\lV u\rV}{\lV x\rV}=1
\end{equation}
and
\begin{equation}\label{e19}
\limsup_{x\to 0}\left[\frac{\lV C(x,0)\rV}{\lV (x,0)\rV}\right]^2=\limsup_{x\to 0}\left[\frac{\lV v\rV}{\lV x\rV}\right]^2\le
1-\liminf_{x\to 0}\left[\frac{f(x)}{\lV x\rV\,\lV \nabla f(x)\rV}\right]^2.
\end{equation}
\end{proposition} 
\begin{proof} 
By convexity of $f$, using \ref{A4}, 
$f(y)\le\langle \nabla f(y),y\rangle\le \lV\nabla f(y)\rV\,\lV y\rV$ for all $y\in\inte(\dom(f))$ sufficiently close to $0$.
Hence, for all nonzero $y\in\inte(\dom(f))$, $0< f(y)/\lV y\rV\le\lV\nabla f(y)\rV$, using \labelcref{A2,A4}.
Since $\lim_{y\to 0}\nabla f(y)=0$ by \labelcref{A2,A3} and the convexity of $f$, it follows that
\begin{equation}\label{e20}{}
\lim_{y\to 0} f(y)/\lV y\rV=0.
\end{equation}
Now we take limits with $x\to 0$ in \eqref{e15}. Since $(u,0)=P_K((x,0))$ and using the continuity of projections,
$\lim_{x\to 0}u=0$. Thus,
\begin{align}
\limsup_{x\to 0}\frac{\lV T(x,0)\rV}{\lV (x,0)\rV} &= \lim_{x\to 0}\frac{\lV T(x,0)\rV}{\lV (x,0)\rV} =\lim_{x\to 0}\frac{1}{\lV \bar u+\dfrac{f(u)}{\lV u\rV}\nabla f(u)\rV} \\
& =\lim_{u\to 0}\frac{1}{\lV \bar u+\dfrac{f(u)}{\lV u\rV}\nabla f(u)\rV}=\frac{1}{\lV \bar u\rV}=1,
\end{align}
using \eqref{e20} and the fact that $\lV\bar u\rV = \lV u/\norm{u}\rV = 1$. We have proved that \eqref{e18} holds.
Now we deal with \eqref{e19}. Taking  $x\to 0$ in \eqref{ee16}, we have
\begin{equation}\label{e21}
\limsup_{x\to 0}\left[\frac{\lV C(x,0)\rV}{\lV (x,0)\rV}\right]^2\le\left[1-\liminf_{x\to 0}\left(\frac{f(u)}{\lV u\rV\,\lV\nabla f(u)\rV}\right)^2\right]
\limsup_{x\to 0}\left[\frac{\lV u\rV}{\lV x\rV}\right]^2.
\end{equation}
The second $\limsup$ on the right-hand side of  \eqref{e21} is equal to $\lim_{x\to 0}\left[\frac{\lV u\rV}{\lV x\rV}\right]^2$ and by \eqref{e18} it is equal to $1$ and so \eqref{e19} follows from the 
already made observation that $\lim_{x\to 0}u=0$.

\qed \end{proof}

We proceed to establish the convergence rates of the sequences generated by MAP and CRM for this choice of $K$ and $U$.

\begin{corollary}\label[corollary]{c2} Consider $K, U\subset\re^{n+1}$ given by $K= \epi(f)$, with 
$f:\re^n\to\re\cup\{+\infty\}$ satisfying
\labelcref{A2,A3,A4} and $U\coloneqq\{(x,0)\mid x\in\re^n\}\subset\re^{n+1}$. Let  $(z^k,0)_{k\in\na}$ and $(x^k,0)_{k\in\na}$, be the sequences generated by MAP and CRM, starting 
from $(z^0,0)\in\re^{n+1}$ and $(x^0,0) \in U$, respectively. Then,
\begin{equation}\label{e22}
\limsup_{k\to\infty}\frac{\lV (z^{k+1},0)\rV}{\lV (z^k,0)\rV}=1
\end{equation}
and
\begin{equation}\label{e23}
\limsup_{k\to\infty}\frac{\lV (x^{k+1},0)\rV}{\lV (x^k,0)\rV}\le\sqrt{1-\gamma^2},
\end{equation}
with 
\begin{equation}\label{e24}
\gamma\coloneqq \liminf_{x\to 0}\frac{f(x)}{\lV x\rV\,\lV\nabla f(x)\rV}.
\end{equation}
\end{corollary}

\begin{proof}
Since $T(z^k,0)=(z^{k+1},0)$, $C(x^k,0)=(x^{k+1},0)$, and $\lim_{k\to\infty}x^k=\lim_{k\to\infty}z^k=0$,
it suffices to apply  \cref{p6} with  $x=z^k$ in \eqref{e18} and $x=x^k$ in \eqref{e19}.

\qed \end{proof}

We add now an additional hypothesis on $f$.
\begin{enumerate}[resume*]
	\item \label{A5} $f$ satisfies $
	\displaystyle\liminf_{x\to 0}\dfrac{f(x)}{\lV x\rV\,\lV\nabla f(x)\rV}>0.$
\end{enumerate} 

Observe that by using the convexity of $f$ and Cauchy-Schwarz inequality,  we have
\[
0\leq f(x)\leq \scal{\nabla f(x)}{x} \leq \norm{\nabla f(x)}\norm{x}.
\]
Thus, \labelcref{A2,A3,A4} imply $f(x)/(\lV x\rV\,\lV\nabla f(x)\rV)\in (0,1]$, for all $x\ne 0$, so that
\ref{A5} just excludes the case in which the $\liminf$ above is equal to $0$. Next we rephrase \cref{c2}.

\begin{corollary}\label[corollary]{c3}
Consider $K, U\subset\re^{n+1}$ given by $K= \epi(f)$, with 
$f:\re^n\to\re\cup\{+\infty\}$ satisfying
\labelcref{A2,A3,A4} and $U\coloneqq\{(x,0)\mid x\in\re^n\}\subset\re^{n+1}$. Then the sequence generated by MAP from an arbitrary initial 
point converges sublinearly. If $f$ also satisfies hypothesis \ref{A5}, then the sequence generated by CRM
from an initial point in $U$ converges linearly, and its asymptotic constant is bounded above by $\sqrt{1-\gamma^2}<1$,
with $\gamma$ as in \eqref{e24}.
\end{corollary}
\begin{proof}
Immediate from \cref{c2} and hypothesis \ref{A5}.

\qed \end{proof}
 
Next we discuss several situations for which hypothesis \ref{A5} holds, showing that it is rather generic. The first case is as
follows.

\begin{proposition}\label[proposition]{p7} 
Assume that $f$, besides satisfying  \labelcref{A2,A3,A4}, is of class ${\cal C}^2$ and
$\nabla^2f(0)$ is nonsingular. Then, assumption \ref{A5} holds, and $\gamma\ge\lambda_{\min}/(2\lambda_{\max})>0$ where
$\lambda_{\min}, \lambda_{\max}$ are the smallest and largest eigenvalues of $\nabla f^2(0)$, respectively.
\end{proposition}

\begin{proof}
In view of \ref{A3}, \ref{A4} and the hypothesis on  $\nabla f^2(0)$, we have
\begin{equation}\label{e25}
f(x)=\frac{1}{2}\scal{x}{\nabla^2f(0)x}+o(\lV x\rV^2)\ge\frac{\lambda_{\min}}{2}\lV x\rV^2+o(\lV x\rV^2).
\end{equation}
Also, using the Taylor expansion of $\nabla f$ around $x=0$, 
$\nabla f(x)=\nabla^{2}f(0)x + o(\lV x\rV)$,
so that 
\begin{align}
\lV x\rV\,\lV\nabla f(x)\rV & =\lV x\rV\,\lV\nabla^2f(0)x\rV +o(\lV x\rV^2) \\
& \le
\lV x\rV^2\lV\nabla^2f(0)\rV +o(\lV x\rV^2)\\
& \le \lambda_{\max}\lV x\rV^2 +o(\lV x\rV^2).\label{e26}
\end{align}
By \eqref{e25}, \eqref{e26},
\begin{equation}\label{e27}
\frac{f(x)}{\lV x\rV\,\lV\nabla f(x)\rV}\ge\frac{\lambda_{\min}\lV x\rV^2+o(\lV x\rV^2)}{2\lambda_{\max}\lV x\rV^2+
o(\lV x\rV^2)}
\end{equation}
and the result follows by taking $\liminf$ in \eqref{e27},  since the right hand converges to $ \dfrac{\lambda_{\min} }{2\lambda_{\max}}>0$  
as $x \to 0$.
\qed  
\end{proof}
 
Note that nonsingularity of $\nabla^2f(0)$ holds when $f$ is of class ${\cal C}^2$ and strongly convex.

We consider next other instances for which assumption \ref{A5} holds. Now we deal with the case in which
$f(x)=\phi(\lV x\rV)$ with $\phi:\re\to\re\cup\{+\infty\}$, satisfying \labelcref{A2,A3,A4}. This case has a one 
dimensional flavor, and computations are easier. The first point to note is that 
\begin{equation}\label{ee27}
\liminf_{x\to 0}\frac{f(x)}{\lV x\rV\,\lV\nabla f(x)\rV}=\liminf_{t\to 0}\frac{\phi(t)}{t\phi'(t)},
\end{equation} 
so that assumption \ref{A5}
becomes:
\begin{enumerate}[label=A\arabic*$'$.,font=\bf,leftmargin=*,ref=A\arabic*$'$,start=4]
	\item \label{A6} $\phi$ satisfies $\displaystyle\liminf_{t\to 0}\dfrac{\phi(t)}{t\phi'(t)}>0$.
\end{enumerate}

More importantly, in this case $\nabla f(x)$ and $x$ are collinear, which allows for an improvement 
in the asymptotic constant: we will have $1-\gamma$ instead of $\sqrt{1-\gamma^2}$ in \eqref{e23}, as we show next.
We reformulate  \cref{p5,p6} for this case.

\begin{proposition}\label[proposition]{p8} 
Assume that $f(x)=\phi(\lV x\rV)$, with $\phi:\re\to\re\cup\{+\infty\}$ satisfying \labelcref{A2,A3,A4} and \labelcref{A6}. 
Take $(x,0)\in U$ with $x\ne 0$. Let $C(x,0)=(v,0)$. Then,
\begin{enumerate}[(i), format=\bf,leftmargin=*,align=right, widest=ii]
\item 
\begin{equation}\label{e28}
\frac{\lV C(x,0)\rV}{\lV T(x,0)\rV}=\frac{\lV v\rV}{\lV u\rV}= 
1-\frac{\phi(\lV u\rV)}{\phi'(\lV u\rV)\lV u\rV},
\end{equation}
\item 
\begin{equation}\label{e29}
\limsup_{x\to 0}\frac{\lV C(x,0)\rV}{\lV (x,0)\rV}=
1-\liminf_{x\to 0}\frac{f(x)}{\lV x\rV\,\lV\nabla f(x)\rV}=
1-\liminf_{t\to 0}\frac{\phi(t)}{t\phi'(t)}.
\end{equation}
\end{enumerate}
\end{proposition} 

\begin{proof}
In this case
\[
\nabla f(x)=\frac{\phi'(\lV x\rV)}{\lV x\rV}x
\]
so that \eqref{e12} becomes 
\[
x=\left(1+\frac{\phi(\lV u\rV )\phi'(u)}{\lV u\rV}\right)u,
\]
and \eqref{e14} can be rewritten as 
\[
v=\left(1-\frac{\phi(\lV u\rV)}{\phi'(\lV u\rV)\lV u\rV}\right)u.
\] 
Hence,
\[
\frac{\lV v\rV}{\lV u\rV}=1-\frac{\phi(\lV u\rV)}{\phi'(\lV u\rV)\lV u\rV},
\]
establishing \eqref{e28}. Then, \eqref{e29} follows from \eqref{e28} as in the proofs
of  \cref{p5,p6}, taking into account \eqref{ee27}. 

\qed \end{proof}

\begin{corollary}\label[corollary]{c4}
Let $(x^k,0)_{k\in\na}$ be the sequence generated by CRM with $(x^0,0)\in U$.
Assume that
$f(x)=\phi(\lV x\rV)$ with $\phi:\re\to\re\cup\{+\infty\}$ satisfying \labelcref{A2,A3,A4}. Then, 
\[\limsup_{k\to\infty}\frac{ \lV (x^{k+1},0)\rV}{\lV (x^k,0)\rV}=1-\hat\gamma,\]
with 
\begin{equation}\label{e30}
\hat\gamma\coloneqq \liminf_{t\to 0}\frac{\phi(t)}{t\phi'(t)}.
\end{equation}
If $\phi$ satisfies hypothesis \ref{A6}
then, the CRM sequence is Q-linearly convergent, with asymptotic constant equal to
$1-\hat\gamma$.
\end{corollary}

\begin{proof}
It is an immediate consequence of \cref{p8}(ii), in view of the definition of the circumcenter operator $C$, given in \eqref{e1}.

\qed \end{proof}

We verify next that assumption \ref{A6} is rather generic. It holds, \emph{e.g.}, if $\phi$ is analytic around $0$.

\begin{proposition}\label[proposition]{p9}
If $\phi$ satisfies \labelcref{A2,A3,A4} and is analytic around $0$ then it satisfies \ref{A6}, and $\hat\gamma= 1/p $,
where $p \coloneqq \min\{j\mid \phi^{(j)}(0)\ne 0\}$.
\end{proposition}

\begin{proof}
In this case $\phi(t)=(1/p !)\phi^{(p )}(0)t^p +o(t^{p +1})$ and $t\phi'(t)=(1/(p -1)!)\phi^{(p )}(0)t^p +
o(t^{p +1})$, and the result follows taking limits with $t\to 0$, taking into account \eqref{e30}.

\qed \end{proof}

Note that for an analytic $\phi$ the asymptotic constant is always of the form $1-1/p$ with $p\in\na$.
This is not the case in general. Take, \emph{e.g.}, $\phi(t)=\lv t\rv^{\alpha}$ with $\alpha\in\re$, $\alpha> 1$.
Then a simple computation shows that $\hat\gamma=1/\alpha$. Note that $\phi$ is of class ${\cal C}^p$,
where $p$ is the integer part of $\alpha$, but not of class ${\cal C}^{p+1}$, so that \cref{p9} does not apply.

Take now
\[ 
f(x)=\begin{cases} 
1-\sqrt{1-\lV x\rV^2},\quad\,\,\, {\rm if}\,\, \lV x\rV \le 1,\\
+\infty,\qquad\qquad\qquad\;\;\; {\rm otherwise},
\end{cases}
\]
\emph{i.e.}, $f(x)=\phi(\lV x\rV)$ with $\phi(t)=1-\sqrt{1-t^2}$, when $t\in [-1,1], \phi(t)=+\infty$ otherwise. 
Note that $f$ satisfies \labelcref{A2,A3,A4} and its effective domain is the 
unit ball in $\re^n$. Since $\phi$ is analytic around $0$ and $\phi''(0)\ne 0$, we get 
from \cref{p9} that $\hat\gamma=1/2$ and so the asymptotic constant of
the CRM sequence is also $1/2$. Note that the graph of $f$ is the lower hemisphere of the ball $B\subset\re^{n+1}$ centered
at $(0,1)$ with radius $1$. Observe also that the projection onto $B$ of a point of the form $(x,0)\in\re^{n+1}$
is of the form $(u,t)$ with $t<1$, so it belongs to $\epi(f)$. Hence, the sequences generated by CRM for the pair
$K,U$ with $K=\epi(f)$ and $K=B$ coincide. It follows easily that the sequence generated by CRM for a pair
$K,U$ where $K$ is any ball and $U$ is a hyperplane tangent to the ball, converges linearly, with asymptotic constant equal
to $1/2$.  
We remark that in all these cases the sequence generated by MAP converges sublinearly, by virtue of \cref{c3}.

We look now at a case where hypothesis \ref{A6} fails. Define
\[
f(x)=\begin{cases} 
e^{-\lV x\rV^{-2}},\quad\,\, {\rm if}\,\, \lV x\rV \le \frac{1}{\sqrt 3},\\
+\infty,\qquad\quad\,\,\, {\rm otherwise}.
\end{cases}
\]
so that $f(x)=\phi(\lV x\rV)$ with $\phi(t)=e^{-1/t^2}$, when $t\in (-3^{-1/2}, 3^{-1/2}), \phi(t)=+\infty$ otherwise.
Again $f$ satisfies \labelcref{A2,A3,A4}. It is easy to check that $\phi(t)/(t\phi'(t))=(1/2)t^2$, so that
$\lim_{t\to 0}\phi(t)/(t\phi'(t))=0$ and \ref{A6} fails. It is known that this $\phi$, which is of class ${\cal C}^\infty$
but not analytic, is extremely flat (in fact, $f^{(k)}(0)=0$ for all $k$), and not even CRM can overcome so much flatness;
in view of \cref{c4}, in this case it converges sublinearly, as MAP does. The examples above are also  presented as a study case in \cite{Bauschke:2016}, illustrating the slow convergence of the proximal point algorithm, Douglas-Rachford algorithm and alternating projections. 

Let us abandon 
such an appalling situation, and move over to other 
examples where CRM will be able to exhibit again its superiority; next, we deal with our second family of examples. 
In this case we keep the framework of the first family with just one change, namely in hypothesis \ref{A4} on $f$; now
we will request that $f(0)<0$. With this single trick  (and a couple of additional technical assumptions),
we will achieve linear convergence of the  MAP sequence and superlinear convergence of the CRM one.
We will assume also that the effective domain of $f$ is the whole space (differently from the previous section, we don't 
have now interesting examples with smaller effective domains; also, since now the limit of the sequences can be 
anywhere, a hypothesis on the effective domain becomes rather cumbersome). We'll also demand that $f$ be of class 
${\cal C}^2$. 

Finally, we will restrict ourselves to the case of $f(x)=\phi(\lV x\rV)$, with $\phi:\re\to\re\cup\{+\infty\}$. 
This assumption is not essential, but will considerably simplify our analysis. Thus, we rewrite the assumptions
for $\phi$, in this new context. We assume that function $\phi$  is proper, strictly convex and twice continuously differentiable, satisfying
\begin{enumerate}[label=A\arabic*$^\prime$.,font=\bf,leftmargin=*,ref=A\arabic*$^\prime$,start=2]
\item \label{A3prime}$\phi'(0)=0$.
\item \label{A4prime}$\phi(0)<0$.
\end{enumerate}

In the remainder of the paper we will study the behavior of the MAP and CRM sequences 
for the pair $K,U\subset\re^{n+1}$, where
$K$ is the epigraph of $f(x)=\phi(\lV x\rV)$, with $\phi$ satisfying hypotheses \labelcref{A3prime,A4prime}  above,
and $U\coloneqq\{(x,0)\mid x\in\re^n\}\subset\re^{n+1}$. As in the previous case,  \cref{p2,p3} ensure that
both sequences converge to points in $K\cap U$. Since we are dealing with convergence rates, 
we will exclude the case 
in which the sequences of interest have finite convergence. We continue with an elementary property of
the limit of these sequences.

\begin{proposition}\label[proposition]{p10} Assume that $K,U$ are as above.  
Let $(x^*,0)$ be the limit of either the MAP or the CRM sequences and $t^*\coloneqq \lV x^*\rV$. Then, 
$\phi(t^*)=0$ and $\phi'(t^*)>0$.
\end{proposition}
\begin{proof}
Since these sequences stay in $U$,
remain outside $K$ (otherwise convergence would be finite), and converge to points in
$K\cap U$, it follows that their limits must belong to $\bound(K)\cap U$, where $\bound(K)\coloneqq \{(x,f(x))\mid x\in\re^n\}$ 
denotes the boundary of $K$. So, we conclude that
$0=f(x^*)=\phi(t^*)$. Now, since $\phi'(0)=0$,  in view of \labelcref{A3prime},  and $\phi'$ is strictly increasing, we
conclude that $\phi'(t)>0$ for all $t>0$. Note that $x^*\ne 0$, because $f(x^*)=0$ and $f(0)<0$ by \ref{A4prime}.
Hence $t^*=\lV x^*\rV>0$, so that $\phi'(t^*)>0$.

\qed \end{proof}   
 
Now we analyze the behavior of the operators $C$ and $T$, in this case.

\begin{proposition}\label[proposition]{p11} 
Assume that  $K,U\subset\re^{n+1}$ are defined as $U\coloneqq\{(x,0)\mid x\in\re^n\}\subset\re^{n+1}$ and 
$K= \epi(f)$ where $f(x)=\phi(\lV x\rV)$ and $\phi$ satisfies \labelcref{A3prime,A4prime}. Let $T$ and $C$ be the operators
associated to MAP and CRM respectively, and $(z^*,0)$ and $(x^*,0)$ the limits of the sequences $(z^k)_{k\in\na}$ and $(x^k)_{k\in\na}$ 
generated
by these methods, starting from some $(z^0,0)\in\re^{n+1}$, and some $(x^0,0)\in U$, respectively. Then,
\begin{equation}\label{e37}  
\limsup_{x\to z^*}\frac{\lV T(x,0)-(z^*,0)\rV}{\lV (x,0)-(z^*,0)\rV}=
\frac{1}{1+\phi'(\lV z^*\rV)^2}
\end{equation}
and
\begin{equation}\label{e38}
\limsup_{x\to z^*}\frac{\lV C(x,0)-(x^*,0)\rV}{\lV (x,0)-(x^*,0)\rV}=0.
\end{equation}
\end{proposition}

\begin{proof} 
Since, in this case, $\nabla f(x)=\dfrac{\phi'(\lV x\rV)}{\lV x\rV}x$ for all $x\ne 0$, we rewrite \eqref{e12} and 
\eqref{e14} as
\begin{equation}\label{e31}
x=\left(1+\frac{\phi(\lV u\rV)\phi'(\lV u\rV)}{\lV u\rV}\right)u
\end{equation}
and
\begin{equation}\label{e32} 
v=\left(1-\frac{\phi(\lV u\rV)}{\phi'(\lV u\rV)\lV u\rV}\right)u.
\end{equation}
In view of \eqref{e31} and \eqref{e32}, $u,v$ and $x$ are collinear. In terms of the operators
$C$ and $T$, we have that $x,C(x)$ and $T(x)$ are collinear, so the same holds for the whole sequences
generated by MAP, CRM  and hence also for their limits $(z^*,0), (x^*,0)$. This is a consequence of 
the one-dimensional flavor of this family of examples. So, we define $s\coloneqq\lV z^*\rV$, $t\coloneqq\lV x^*\rV$,
$r\coloneqq\lV u\rV$, and therefore we get $u=(r/s)z^*=(r/t)x^*$. We compute next the   
quotients 
\[\frac{\lV (T(x),0)-(z^*,0)\rV}{\lV (x,0)-(z^*,0)\rV} = \frac{\lV u-z^*\rV}{\lV x-z^*\rV}
\]
and 
\[\frac{\lV (C(x),0)-(x^*,0)\rV}{\lV (x,0)-(x^*,0)\rV} =
\frac{\lV v-x^*\rV}{\lV x-x^*\rV},
\] needed for determining the convergence rate of the MAP and CRM sequences.
We start with the MAP case.
\begin{align}
\frac{\lV T(x,0)-(z^*,0)\rV}{\lV (x,0)-(z^*,0)\rV} &=  \frac{\lV u-z^*\rV}{\lV x-z^*\rV}=
\frac{s\lv\frac{r}{s}-1\rv}{s\lv\frac{r}{s}-1+\phi(r)\phi'(r)\rv}
\\
&=\frac{\lv r-s\rv}{\lv r-s+s\phi'(r)\phi(r)\rv}\\
&=\frac{1}{\lv 1+s\phi'(r)\left(\frac{\phi(r)-\phi(s)}{r-s}\right)\rv},\label{e33}
\end{align}
using \eqref{e31} in the second equality and the fact that $s=\phi(\lV z^*\rV)=f(z^*)=0$, established in \cref{p10}, in the fourth one.

Now, we perform a similar computation for the operator $C$, needed for the CRM sequence.

\begin{align}
\frac{\lV C(x,0)-(x^*,0)\rV}{\lV (x,0)-(x^*,0)\rV} &= \frac{\lV v-x^*\rV}{\lV x-x^*\rV}=
\frac{t\lv \left(1-\frac{\phi(r)}{\phi'(r)r}\right)\frac{r}{t}-1\rv}{t\lv\frac{r}{t}-1+\phi(r)\phi'(r)\rv}
\\
&=\frac{\lv\left(1-\frac{\phi(r)}{r\phi'(r)}\right)r-t\rv}{\lv r-t+t\phi(r)\phi'(r))\rv}=
\frac{\lv r-t-\frac{\phi(r)}{\phi'(r)}\rv}{\lv r-t+t\phi(r)\phi'(r)\rv}
\\
\label{e34}
&=\frac{\lv 1-\frac{1}{\phi'(r)}\left(\frac{\phi(r)-\phi(t)}{r-t}\right)\rv}{\lv 1+
t\phi'(r)\left(\frac{\phi(r)-\phi(t)}{r-t}\right)\rv},
\end{align}
using \eqref{e32} in the second equality, and \cref{p10}, which implies $\phi(t)=0$, in the fifth one.

Finally, we take limits in \eqref{e33} with $x\to z^*$ and in \eqref{e34} with $x\to x^*$. Note that, since $u=P_K(x)$,
$\lim _{x\to z^*}u=P_K(z^*)=z^*$, because $z^*\in K$. Hence we take limit with $r\to s$ in the right hand side of
\eqref{e33}. We also take limits with $x\to x^*$ in \eqref{e34}. By the same token, taking limit with $r\to t$ in the right hand side,
we get
\begin{align}
\limsup_{x\to z^*}\frac{\lV T(x,0)-(z^*,0)\rV}{\lV (x,0)-(z^*,0)\rV} &=
 \limsup_{r\to s}\frac{1}{\lv 1+s\phi'(r)\left(\frac{\phi(r)-\phi(s)}{r-s}\right)\rv}\\ &= \frac{1}{\lv 1+s\displaystyle\lim_{r\to s}\phi'(r)\left(\frac{\phi(r)-\phi(s)}{r-s}\right)\rv} \\ &=
\frac{1}{1+s\phi'(s)^2} \label{e35}
\end{align}
and 
\begin{align}\label{e36}
\limsup_{x\to x^*}\frac{\lV C(x,0)-(x^*,0)\rV}{\lV (x,0)-(x^*,0)\rV}  & = 
\limsup_{r\to t}\frac{\lv 1-\frac{1}{\phi'(r)}\left(\frac{\phi(r)-
\phi(t)}{r-t}\right)\rv}{\lv 1+t\phi'(r)\left(\frac{\phi(r)-\phi(t)}{r-t}\right)\rv}\\
& = 
\frac{\lv 1-\displaystyle\lim_{r\to t}\frac{1}{\phi'(r)}\left(\frac{\phi(r)-\phi(t)}{r-t}\right)\rv}{\lv 1+t\displaystyle\lim_{r\to t}\phi'(r)\left(\frac{\phi(r)-\phi(t)}{r-t}\right)\rv}\\
&=\frac{\lv 1-\frac{\phi'(t)}{\phi'(t)}\rv}{\lv 1+t\phi'(t)^2\rv}=0. 
\end{align}
The results follow, in view of the definitions of $s$ and $t$, from \eqref{e35} and \eqref{e36}, respectively.

\qed \end{proof}

Note that the denominators in the expressions of \eqref{e35} and \eqref{e36} are the same; the difference lies
in the numerators: in the MAP case it is $1$; in the CRM one, the presence of the factor $(\phi(r)-\phi(t))/(r-t)$
makes the numerator go to $0$ when $r$ tends to $t$.
 
\begin{corollary}\label[corollary]{c5}
Under the assumptions of \cref{p11} the sequence generated by MAP converges 
Q-linearly to a point $(z^*,0)\in K\cap U$, with
asymptotic constant equal to  $1/(1+\phi'(\lV z^*\rV)^2)$, and the sequence
generated by CRM converges superlinearly.
\end{corollary}

\begin{proof}
The result for the MAP sequence follows from \eqref{e37} in \cref{p11}, observing
that for $x=z^k$, we have $T(x,0)=(z^{k+1}, 0)$. Note that the asymptotic constant is indeed
smaller than $1$, because $z^*\ne 0$, and $\phi'(\lV z^*\rV)\ne 0$ by \cref{p10}.
The result for the CRM sequence follows from \eqref{e38} in \cref{p11}, observing that for $x=x^k$,
we have $C(x,0)=(x^{k+1},0)$.

\qed \end{proof}

We now present an example that, although very simple, enables one to visualize how fast CRM is in comparison to MAP. 

\begin{example}
Let $\phi:\re\to \re$ given by $\phi(t) = \lv t\rv^\alpha - \beta$, where $\alpha>1$ and $\beta \geq 0$.  Consider $K, U\subset \re^2$ such that  $K \coloneqq \epi(\phi)$ and $U$ is the abscissa axis. Note that, if $\beta = 0$, the error bound condition \ref{EB}  between $K$ and $U$ does not hold. For any $\beta >0$, though, it is easily verifiable that  \ref{EB} is valid. \Cref{fig:1} shows CRM and MAP tracking a point in $K\cap U$ up to a precision $\varepsilon > 0$, with the same starting point $(1.1, 0)\in\re^2$.  We fix $\alpha=2$ and take $\beta = 0$ in \Cref{fig:1a} and $\beta = 0.06$ in \Cref{fig:1b}.  We count and display the iterations of the MAP sequence  $(z^k)_{k\in\na}$ and the CRM sequence $(x^k)_{k\in\na}$ until $\dist(z^k,K\cap U) \leq \varepsilon$ and $\dist(x^k,K\cap U) \leq \varepsilon$, with $\varepsilon = 10^{-3}$. The figures below depict the results on MAP and CRM derived in  \cref{c3,c5}.

\begin{figure}[!htb]
\centering
  \begin{subfigure}[b]{\linewidth}
    \centering
  \includegraphics[width=.9\textwidth]{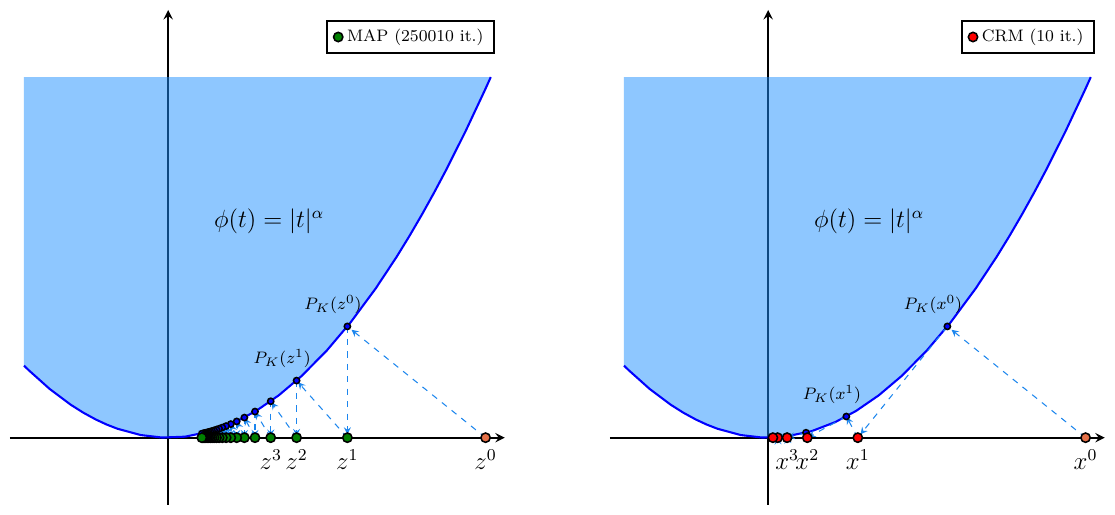}    \caption{\label{fig:1a} Lack of error bound: MAP converges sublinearly and CRM linearly.}
  \end{subfigure}\\
  \begin{subfigure}[b]{\linewidth}
    \centering
    \includegraphics[width=.9\textwidth]{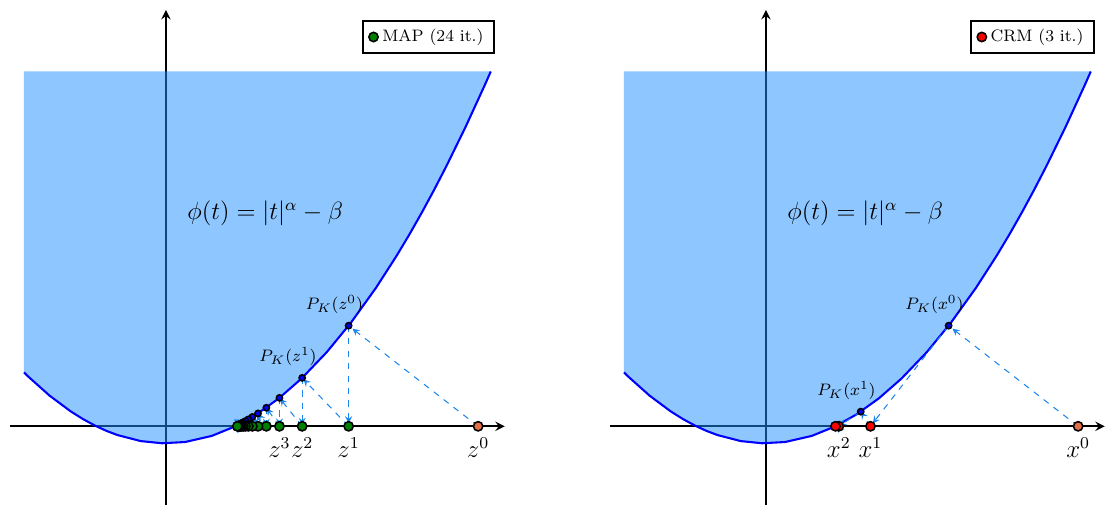}
    \caption{\label{fig:1b} Presence of error bound: MAP converges linearly and CRM superlinearly.}
  \end{subfigure} 
\caption{Illustrative comparison between MAP and CRM.}
\label{fig:1}       
\end{figure}

\end{example}

We emphasize that in the cases above MAP exhibits its usual behavior, \emph{i.e.}, linear convergence.
The examples of the first family were somewhat special because, roughly speaking, the angle between 
$K$ and $U$ goes to $0$ near the intersection. On the other hand, the superlinear convergence of CRM is quite remarkable. 
The additional computations of CRM over MAP reduce to the trivial determination of the reflections and the solution of a system of two linear equations in two  variables, for finding the circumcenter~\cite{Behling:2018a,Bauschke:2018a}. Now MAP is a typical first-order 
method (projections disregard the curvature of the sets), and thus its convergence is generically no better than linear.
We have shown that the CRM acceleration improves this linear convergence to superlinear in a rather large class of instances. Long live CRM!

 We conjecture that CRM enjoys superlinear convergence whenever $U$ intersects the interior of $K$.
The results in this section firmly support this conjecture.

  \begin{acknowledgements}
We thank the anonymous referees for their valuable suggestions which significantly improved  this manuscript. 
\end{acknowledgements}



\end{document}